\documentclass[11pt, final ]{amsart}
\usepackage{amsmath,amsthm,amssymb}

\usepackage[notref]{showkeys}
\usepackage[dvips]{graphicx}

\usepackage{color,epsfig}
\usepackage{graphicx}
\usepackage{color}

\usepackage{bbm}
\usepackage[all]{xy}
\usepackage{amscd}
\usepackage{stmaryrd}
\usepackage{verbatim}

\usepackage[english]{babel}

\theoremstyle{theorem}
\newtheorem{thm}{Theorem}[section]
\newtheorem{lem}[thm]{Lemma}

\newtheorem{prop}[thm]{Proposition}

\theoremstyle{remark}

\theoremstyle{definition}

\numberwithin{equation}{section}

\def\P{{\mathbb{P}}}
\def\R{{\mathbb{R}}}

\newcommand{\EE}{\mathbb{E}}
\newcommand{\PP}{\mathbb{P}}
\newcommand{\E}{\mathbb{E}}
\newcommand{\G}{\mathcal{G}}
\newcommand{\X}{\mathcal{X}}
\newcommand{\Z}{\mathbb{Z}}
\renewcommand{\P}{\mathbb{P}}

\newcommand{\N}{\mathbb{N}}

\newcommand{\F}{\mathcal{F}}
\newcommand{\eps}{\varepsilon}

\newcommand{\8}{\infty}
\renewcommand{\a}{\alpha}
\renewcommand{\b}{\beta}

\newcommand{\wZ}{\widetilde Z}
\newcommand{\wY}{\widetilde Y}

\newcommand{\W}{{\mathbb W }}
\newcommand{\Wbu}{{\mathbb W }^{\uparrow}}
\newcommand{\Wbd}{{\mathbb W }^{\downarrow}}
\newcommand{\Wu}{{ W }^{\uparrow}}
\newcommand{\Wd}{{ W }^{\downarrow}}

\newcommand{\wt}{\widetilde}

\usepackage{soul}

\renewcommand{\o}{\omega}

\begin{document}

	\title[Large deviations for RWRE]{Precise large deviations for random walk in random environment\footnote{\today}}

	\author[D.~Buraczewski, P.~Dyszewski]{Dariusz Buraczewski, Piotr Dyszewski}

	\address{Instytut Matematyczny, Uniwersytet Wroclawski, Plac Grunwaldzki 2/4, 50-384 Wroclaw, Poland}
	\email{dbura@math.uni.wroc.pl, pdysz@math.uni.wroc.pl}

	\thanks{The authors  were partially supported by the National Science Center, Poland (Sonata Bis, grant number DEC-2014/14/E/ST1/00588)}


	\begin{abstract}
		We study one-dimensional nearest neighbour random walk in site-random environment. We establish precise (sharp) large deviations
		in the so-called ballistic regime, when the random walk drifts to the right with linear speed. In the sub-ballistic regime, when the speed is 
		sublinear, we describe the precise probability of slowdown.
	\end{abstract}

	\maketitle

\section{Introduction}

\subsection{ Random walk in random environment}
	Throughout this article we will be interested in some asymptotic properties of nearest neighbour random walk in
	site-dependent random medium. Starting from the early work of Solomon~\cite{solomon1975random}, this model has attracted a lot of attention over the past few years since,
	apart from motivations originated in physics, it exhibits a lot of features not
	observed in the classical random walk. We refer to the notes of Zeitouni~\cite{zeitouni2004part} for an introduction to the topic.

	The main contribution of this article is an extension of large deviation results obtained previously by Dembo, Peres and Zeitouni~\cite{dembo1996tail} to precise (rather than logarithmic) 	
	asymptotic of the deviations. We establish also precise probability of slowdown, when the speed of the random walk is sublinear, improving thus the result of Fribergh, Gantert and Popov~\cite{fribergh2010slowdown}.
For a precise set-up, let $\Omega = (0,1)^\Z$ be the set of all possible configurations of the environment and let $\F$ be the
	$\sigma$-algebra generated by the cylindrical  subsets of the product space $\Omega$. An  environment is an element
	$\o = (\o_n)_{n\in \Z}$ of the measurable space $(\Omega,\F)$. By $P$ we denote a probability distribution on
	$(\Omega, \F)$. Once the environment $\o$ is chosen with respect to $P$ it remains fixed and determines the transition
	kernel of a random walk starting at point $0$. Denote the set of trajectories by $\X = \Z^\N$ and let $\G$ be
	the corresponding $\sigma$-algebra.
	A quenched (fixed) environment $\o$ provides us with a random probability measure $\P_\o$ on $\X$, such that $\P_\o (X_0=0)=1$ and
	$$
		\P_\o(X_{n+1}=j| X_n = i) = \left\{
			\begin{array}{cl}
			\o_i & \mbox{if } j=i+1,\\
			1- \o_i\ & \mbox{if } j=i-1,\\
			0 & \mbox{otherwise.}
		\end{array}\right.
	$$
	Then $X=(X_n)_{n\ge 0}$ is a Markov chain on $\Z$ (with respect to $\P_\o$), called random walk in  random
	environment $\o$ (RWRE).\medskip

	In the context of RWRE one can distinguish two equally valid aspects, that is quenched and annealed behaviour. The former refers to phenomena encountered with respect to
	$\P_\o$ for almost all (a.a.) $\o$. The latter, which is our main focus here, is with respect to the annealed probability, that is the average of $\P_\o$ over $\o$.
	Formally, we define the annealed probability $\P$ as follows. By monotone class theorem, one can verify the measurability of the map $\o \mapsto \P_\o(G)$ for any $G \in \G$.
	This allows us define the mentioned  annealed probability measure $\P$ on $(\Omega\times \X, \F\otimes \G)$, which is a semi-direct product $\P = P \ltimes \P_\o$ given by
	$$
		\P(F\times G) = \int_F \P_\o(G) P(d\o),\qquad F\in \F, G\in \G.
	$$
	Note that $X$ does not form a Markov Chain under the annealed measure $\P$ since, loosely speaking,  the process $X$
	"learns" the environment as it traverses $\Z$.  Thought this article we will assume a particular structure of the environment, namely
	that the measure $P$ on $\Omega$ is chosen is such a way that $\o = (\o_n)_{n \in \Z}$ forms a sequence of independent
	identically distributed (iid) random variables.

\medskip

	One natural question regarding the behaviour of $X$ concerns limit theorems analogous to those treating classical random walk. Obviously one has to take the random environment
	into account. To quantify it, consider the random variables
	$$
		A_n = \frac{1-\o_n}{\o_n}, \quad n \in \Z.
	$$
	This sequence will play a crucial role in what follows, since $A_n$'s are the means of a reproduction laws of a branching process associated with $X$ (see Section~\ref{sec:brrei} for details).
	Solomon~\cite{solomon1975random} proved that the process $X$ is $\o$ a.s. transient if and only if $\E\log A\not = 0$. Here we
	are interested in the transient case when
	\begin{equation} \label{eq:trans}
		\E\log A<0
	\end{equation}
	 and then, since the environment prefers a jump to the right,
	$\lim_{n\to\8} X_n = +\8$ $\P$ a.s. Solomon~\cite{solomon1975random} proved also the law of large numbers, that is $\PP$ a.s.
	\begin{equation}\label{eq:LLN}
		\lim_{n\to\8} \frac{X_n}n= v .
	\end{equation}
	It is known that the limit $v$ is constant $\P$ a.s. and that one can distinguish two regimes
	\begin{enumerate}
		\item \label{fast} ballistic regime ($\E A <1$), when $v = \frac{1-\E A}{1+\E A}$,
		\item \label{slow} sub-ballistic regime ($\E A \geq 1$),  when $v=0$.
	\end{enumerate}
	The first order asymptotic of $X$ in the recurrent case was investigated by Sinai~\cite{sinai1982limit} with a weak limit identified by Kesten~\cite{kesten1986limit}.
	The central limit theorem corresponding to~\eqref{eq:LLN} was proved by Kesten, Kozlov and Spitzer~\cite{kesten1975limit} yielding a weak convergence of
	$$
		\frac{X_n-vn}{a_n(\alpha)}.
	$$
	The limiting distribution as well as the appropriate normalization $a_n(\alpha)$ are related to the value of a parameter $\alpha>0$, for which 	
	\begin{equation}\label{eq:kesten}
		\E A^{\a} = 1.
	\end{equation}
	Note that the above condition for $\a>1$ implies ballisticity.

\subsection{ The ballistic regime}	The aim of this article is to investigate large deviations corresponding to the convergence~\eqref{eq:LLN}. This problem already attracted some attention in the probabilistic community
	resulting in works of  Dembo et. al~\cite{dembo1996tail}, Pisztora, Povel and Zeitouni~\cite{pisztora1999precise} and Varadhan~\cite{varadhan2003large}. However all  mentioned
	articles deliver asymptomatic of the logarithm of probability of a large deviation. Our aim is to sharpen some of this results and deliver a (precise) asymptotic of probability of a large
	deviation.

	The quenched behaviour, which is not of our interest here, also accumulated a fair amount of literature devoted to it. This resulted in the works of
	Greven and den Hollander~\cite{greven1994large}, Gantert and Zeitouni~\cite{gantert1998quenched}, Comets, Gantert and Zeitouni~\cite{comets2000quenched} and
	Zerner~\cite{zerner1998lyapounov}.
	In spite of the time that had passed since the work of  Solomon~\cite{solomon1975random}, RWRE sill attract a lot of attention in the literature as seen from the research of
	 Dolgopyat and Goldsheid~\cite{dolgopyat2012quenched}, Peterson, Jonathon and Samorodnitsky \cite{peterson2013weak},
	 Bouchet, Sabot and dos Santos~\cite{bouchet2016quenched}.

\medskip

	In this paper we consider large deviations of $\frac{X_n}n$ in the ballistic regime
    and aim to describe asymptotic behaviour of
	$\P(X_n - vn < -x)$ as $n,x \to \infty$. We assume only that $\P[A>1]>0$ which, excluding some degenerate cases,
	entails \eqref{eq:kesten} for some $\alpha >0$. In regime~\eqref{fast} this problem was considered by Dembo et al.~\cite{dembo1996tail} where it was established
	that the probability of a deviation is subexponential.

	\begin{lem}[Dembo, Peres, Zeitouni \cite{dembo1996tail}]\label{lem: DPZ}
		Assume that $A$ is bounded a.s., $\P(A =1)<1$ and that~\eqref{eq:kesten} is satisfied for some $\a>1$.
		Then for any open $G\subset (0,v)$ separated from $v$,
		$$
			\lim_{n\to\8}\frac{\log \P(n^{-1}X_n \in G)}{\log n} = 1-\a.
		$$
	\end{lem}

	We aim to prove a result treating a precise behaviour of deviations of $X$ rather than logarithmic.

\begin{thm}\label{thm:main1}
	Suppose that~\eqref{eq:kesten} holds for some  $\a>1$, $\PP[A=1]<1$ and that $\E A^{\a+{ \delta}}<\8$ for some  $\delta >0$. Assume additionally that
	the law of $\log A$ is nonarithmetic. Then
	\begin{equation}\label{eq:mthm2}
		\lim_{n\to\8} \sup_{x \in \Gamma_n  }\bigg| \frac{\P(X_n -  vn< - x)}{( vn-x)x^{-\a}}  - \mathcal{C}(\a) \bigg| = 0,
	\end{equation}
	where $\mathcal{C}(\a)>0$ and
	\begin{equation*}
  		\Gamma_n = \left \{
  		\begin{array}{ll}
  			\big(n^{1/\a}(\log n)^M,vn -b_n\big)\quad & \mbox{for $\a \in (1,2]$}\\
  			\big( c_n n^{1/2}\log n, vn-b_n\big) \quad & \mbox{for $\a> 2$}\\
  		\end{array}
  		\right.,
	\end{equation*}
  	where $M>2$, $\varepsilon>0$ and $b_n,c_n \to \infty$ such that $c_n \leq n^{1/2}\log(n)^{-1}$ and $b_n <vn- n^{1/\a} \log(n)^M$ if $\alpha \in (1,2]$ and $b_n < vn-c_n n^{1/2}\log(n)$ if $\alpha>2$.
	In particular, choosing $x=\varepsilon n$, 
	$$
		\lim_{n\to\8} \frac{ \P(X_n < (v-\eps)n)}{n^{1-\a}} = (v-\varepsilon)\varepsilon^{-\alpha}\mathcal{C}(\a).
	$$

\end{thm}

	The constant $\mathcal{C}(\alpha)$ can be represented in the terms of branching process with immigration associated with $X$. We will provide more details in Section~\ref{sec:brrei} and
	Section~\ref{sec:approach} after we present the construction of the process in question and deliver some tools.\\

	In order to prove our main result, we will use the fact that jumps of $X$ have a structure of a branching process with immigration. The problem of large deviations of $X$ will boil down to 	deviations of the total population size of mentioned branching process. This approach was used previously by Dembo et al.~\cite{dembo1996tail} and Kesten et
	al~\cite{kesten1975limit}. Next, since the branching process can be relatively well approximated by the environment, we will be able to determine the most probable moment, when the
	deviation happen. A fortiori, the large deviations of $X$ come from large deviations of the environment, which is a phenomena used by Dembo et al.~\cite{dembo1996tail} and
	Kesten et al.~\cite{kesten1975limit}. The final arguments leading us to Theorem \ref{thm:main1} strongly base on the methods developed by Buraczewski et al. \cite{Buraczewski:Damek:Mikosch:Zienkiewicz:2013}, who considered large deviations results for  partial sums of some stochastic recurrence equation.

\subsection{ The sub-ballistic regime}
	If condition \eqref{eq:kesten} holds for some $\alpha \le 1$, then $X_n/n$ converges to 0 a.s. For $\alpha<1$ the process $\{X_n\}$ is typically at distance of order $O(n^{\alpha})$ from the origin, as follows from \cite{kesten1975limit}. 
	The annealed probability of slowdown was described by Fribergh et al. \cite{fribergh2010slowdown}, who proved that it decays polynomially.

	\begin{lem}[Fribergh, Gantert, Popov \cite{fribergh2010slowdown}]\label{lem: FGP}
		Assume that \eqref{eq:kesten} holds for $\alpha \le 1$, $\E[A^{\alpha}\log^+A]<\infty$
		and $\E[A^{-\delta}]<\infty$ for some $\delta>0$. Then for any $\beta\in(0,\alpha)$
		$$
			\lim_{n\to\infty} \frac{\log \P(X_n < n^{\beta})}{\log n} = \beta-\alpha.
		$$
	\end{lem}

	Here we obtain a precise asymptotic.

	\begin{thm}\label{thm:main2}
		Suppose that \eqref{eq:kesten} holds for some $\alpha\le 1$ and $\E[A^{\alpha+\delta}]<\infty$ for $\delta >0$. Assume additionally that the law of $\log A$ is nonarithmetic. Then
		\begin{equation}\label{eq:56}
			\lim_{n\to\8} \sup_{x \in \Gamma_n  }\bigg| \frac{\P(X_n<  x)}{ x n^{-\a}}  - \mathcal{C}(\a) \bigg| = 0,
		\end{equation} 
		where $\mathcal{C}(\alpha) > 0$ and $\Gamma_n = (c_n\log n, n^{\alpha}/(\log n)^M)$ for $M>2\alpha$ and $c_n\to\infty$.

		In particular putting $x=n^{\beta}$ for any $\beta\in(0,\alpha)$, we obtain
		$$
			\lim_{n\to\infty} \frac{\P(X_n < n^{\beta})}{n^{\beta-\alpha}} = \mathcal{C}(\alpha).
		$$
	\end{thm}

\subsection{ The structure of the paper}
	The article is organized as follows. In Section~\ref{sec:brrei} we present an associated branching process in random environment with immigration and translate the problem of large
	deviations of RWRE into those of BPRE with immigration. In Section~\ref{sec:approach} we present some intuitions related to our arguments. The last three sections are devoted to
	the proof of our results.

\section{Branching process in random environment with immigration}\label{sec:brrei}

	From now on, we will suppose that the assumptions of Theorem~\ref{thm:main1} are in force.

\subsection{Construction of associated branching process with immigration}
	We will begin by introducing a branching process in random environment with immigration associated with $X$. For this reason consider the first hitting time of $X$, given viz.
	$$
		T_n = \inf\{  k: X_k = n \}.
	$$
	As shown in \cite{kesten1975limit}, one can express $T_n$ using a branching process. To see that, let $U_i^n$ be the number of steps made by $X$ from $i$ to $i-1$ during
	$[0,T_n)$, that is
	$$
		U_i^n = \# \big\{ k < T_n: X_k=i, X_{k+1} = i-1 \big\}, \qquad i < n.
	$$
	Then, since $X_0=0$ and $X_{T_n}=n$, we have
	\begin{eqnarray*}
		T_n &=& \# \mbox{ of steps during $[0,T_n)$}\\
			&=& \# \mbox{ of steps to the right during $[0,T_n)$} + \# \mbox{ of steps to the left during $[0,T_n)$}\\
			&=& n + 2 \cdot \# \mbox{ of steps to the left during $[0,T_n)$}\\
			&=& n + 2\sum_{i < n} U_i^n.
	\end{eqnarray*}
	Note that the summation above extends over all integers $i \in ( -\infty, n)$. As a conclusion, all the randomness of $T_n$ comes from the infinite sum
	\begin{equation}\label{eq:1}
  		\sum_{i < n} U_i^n.
	\end{equation}
	It turns out that $(U_i^n)_{i\leq n}$ exhibits a branching structure. To make it evident, fix an environment $\o \in \Omega$, an integer $n \geq 0$ and consider the
	sequence $U_n^n, U_{n-1}^n, \ldots$.
	Obviously $U_n^n=0$ since $X$ cannot reach $n$ before the time $T_n$. Firstly, we will inspect $0\le i <n$. Note that a jump  $i \to i-1$ can occur either before the first jump
	$i+1 \to i$, between two jumps $i+1\to i$ or after a last jump $i+1 \to i$. Whence, we may express $U_i^n$ in the following fashion
	$$
		U^n_i = \sum_{k=1}^{U_{i+1}^n}V^i_k + V^i_0, \quad 0\le i< n,
	$$
	where $V^i_0$ denotes the number of jumps $i \to i-1$ before the first jump $i+1 \to i$, for $U^n_{i+1}>k >0$, $V^i_k$ denotes the number of jumps $i \to i-1$ between $k$th and
	$ (k+1)$th jump $i+1 \to i$ and for $k = U_{i+1}^n$ is the number of jumps $i \to i-1$ after the last jump $i+1\to i$. Note that since the underlying random walk is transient to the
	right under $\P_\o$, $V^i_k$'s are iid with geometric distribution with parameter $\o_i$, that is
   	\begin{equation}\label{eq:2}
     		\P_\o(V^i_k = l) = \o_i (1-\o_i)^l
   	\end{equation}
	and moreover there are independent of $U_{i+1}^n$. For $i <0$ the behaviour of $U_i^n$ is different. Since $X$ starts from $0$, there will be no jumps from $i \to i-1$
	before the first jump $i+1 \to i$. Apart from that, the relation between $U_i^n$ and $U_{i+1}^n$ is the same as previously, more precisely
	$$
		U^n_i = \sum_{k=1}^{U_{i+1}^n}V^i_k,  \quad i< 0,
	$$
	where $V_k^i$ is distributed as indicated by~\eqref{eq:2}. In conclusion $\{ U_{n-j}^n \}_{j \geq 0}$ forms a sequence  of generation sizes of an inhomogeneous
	branching process with immigration in which one immigrant enters the system only at first $n$ generations. The reproduction law is geometric with parameter $\o_{n-j}$ in the
	$j$th generation.

\medskip

	We will ease the notation and consider a branching process in random environment $Z=\{Z_{n} \}_{n \geq 0}$ with evolution which can be described as follows. We start at
	time $n=0$ with no particles, so that $Z_0 =0$. Next first immigrant enters  the systems and generates $\xi_0^0$ offspring with geometric distribution with parameter $\o_0$,
	that is
	$$
		\P_\o (\xi_0^0 = l ) = \o_0(1-\o_0)^l,
	$$
	these particles will form the first generation, i.e. $Z_1 = \xi_0^0$. At time $n$ for $ n \geq 1$, $(n+1)$th immigrant enters the system and reproduces independently from other
	particles (with respect to $\P_\o$). Their offspring will form the $(n+1)$th generation, that is
	\begin{equation}\label{eq:inhombr}
		Z_{n+1} = \sum_{k=1}^{Z_{n}} \xi^{n}_k + \xi^{n}_0,
	\end{equation}
	where $\{\xi^n_k\}_{k \geq 0}$ are iid with geometric distribution
	$$
		\P_\o (\xi_0^{n} =l ) = \o_{n}(1-\o_{n})^l
	$$
	and independent of $Z_{n}$. Note that $Z_{n+1}$ depends on the environment up to time $n$, that is it depends on $\o_0$, \ldots $\o_{n}$. To analyse $Z$, it will be convenient
	to group the particles depending on which immigrant they originated from, so let $Z_{i,n}$ denote the number of progeny alive at time $n$ of the $i$th  immigrant. Note that then
	$\{ Z_{i,n} \}_{n\geq i}$ forms a homogeneous branching process, that is $Z_{i,n} = 0$, for $n <i$ and
	$$
		Z_{i,i} \stackrel{d}{=} \xi_{0}^{i-1},
	$$
	with respect to the quenched probability $\P_\o$ for all $\o \in \Omega$, and for $ n>i$,
	\begin{equation}\label{eq:hombr}
		Z_{i,n}  \stackrel{d}{=} \sum_{k=1}^{Z_{i,n-1}} \xi_k^{n-1}.
	\end{equation}
	This process in subcritical, since
	$$
		\E_\o \xi_0^0 = \frac{1-\o_0}{\o_0}= A_0
	$$
	and by our standing assumption $\E[\log(A)] <0$. Whence, we are allowed to consider the total population size of the process initiated by the $i$th immigrant denoted by
	$$
		\wZ^i_{i, \infty} = \sum_{n= i}^{\infty}Z_{i, n}
	$$
	and the total size of population started by the first $n$ immigrants, given via
	$$
		W_n = \sum_{k=1}^{n} \wZ^k_{k, \infty}.
	$$
	Now, since $\o$ for a sequence of iid random variables, after we average over $P$, we can conclude that
	$$
		W_n \stackrel{d}{=} \sum_{i< n}U_i^n \quad \mbox{with respect to } \P.
	$$
	Our strategy is to establish Theorem~\ref{thm:wn} stated below, from which we will infer Theorem~\ref{thm:main1} and Theorem~\ref{thm:main2}.

\begin{thm}\label{thm:wn}
	Under the assumptions of Theorem~\ref{thm:main1} for $\alpha >1$ and
Theorem~\ref{thm:main2} for $\alpha \le 1$
we have
	$$
		\lim_{n\to\8} \sup_{x \in \Lambda_n}\bigg| \frac{\P(W_n - d_n > x)}{nx^{-\a}} - \mathcal{C}_1(\a) \bigg| = 0,
	$$
	where $d_n = \E W_n$ for $\alpha>1$, $d_n=0$ for $\alpha\le 1$ and
	\begin{equation*}
  		\Lambda_n = \left \{
  		\begin{array}{ll}
  			\big(n^{1/\a}(\log n)^M,e^{s_n}\big)\quad & \mbox{for $\a \in (0,2]$}\\
  			\big( c_n n^{1/2}\log n, e^{s_n}\big) \quad & \mbox{for $\a> 2$}\\
  		\end{array}
  		\right.
	\end{equation*}
   	for $M>2$, $c_n$, $s_n \to \infty$ and $s_n = o(n)$.
\end{thm}

Theorems \ref{thm:main1} and \ref{thm:main2} are  relatively simple corollaries  from Theorem~\ref{thm:wn}. Therefore we first establish the implication, and in the remaining part of the paper we concentrate on the proof of the above result.  Below we present how Theorem \ref{thm:main1}
can be deduced. We skip the details concerning our second result, Theorem \ref{thm:main2}.
 From the proof  we can easily deduce that  for the constant $\mathcal{C}(\a)$ appearing in Theorem~\ref{thm:main1} one has
	\begin{equation*}
		\mathcal{C}(\alpha) =  \left\{ \begin{array}{lr}  (2v)^{\alpha}\mathcal{C}_1(\alpha) & \alpha >1 \\ \mathcal{C}_1(\alpha) & \alpha  \leq 1 \end{array} \right. .
	\end{equation*}

\begin{proof}[Proof of Theorem \ref{thm:main1}]
	Recall that with respect to the annealed probability $\PP$
	$$
		T_n \stackrel{d}{=} 2W_n + n,
	$$
	where
	$$
		\E W_n = \frac{n\rho}{1-\rho}, \qquad \rho = \EE A = \lambda(1) <1.
	$$

\medskip

{\sc Step 1. Lower estimates}
	Write for $x \in \Gamma_n$,
	\begin{align*}
 		\P\big( X_n - nv < -x \big) & \ge  \P\big( T_{nv-x}> n \big)\\
  			& =  \P\big( 2 W_{nv-x } +(nv-x) > n \big)\\
  			& =  \P\bigg(  W_{nv-x}  -\E W_{nv-x} >\frac 12 \bigg( n- nv+x -\frac{2\rho (nv-x)}{1-\rho} \bigg) \bigg)\\
  			& =  \P\bigg(  W_{nv-x}  -\E W_{nv-x} >   \frac{x}{2v} \bigg).
	\end{align*}
	Hence
	\begin{align*}
		  \frac {x^\a}{nv-x} \cdot   \P\big( X_n - nv < -x \big) &\ge
		 \frac {x^\a}{nv-x} \cdot  \P\bigg(  W_{nv-x}  -\E W_{nv-x} >   \frac{x}{2v} \bigg) \\&= (2v)^\a \mathcal{C}_1(\a)+o(1)=\mathcal{C}(\alpha) +o(1)
	\end{align*}
	uniformly in with respect to $x \in \Gamma_n$.

{\sc Step 2. Upper estimates}
	We will apply an argument similar to the one presented in~\cite{dembo1996tail}. Denote
	$$
		L_j  = \max_i\big\{ j-X_i:\; i\ge T_j  \big\}
	$$
	to be the longest excursion of $X$ to the left of of $j$, after the first hitting time at $j$. By the virtue of Lemma 2.2 in~\cite{dembo1996tail}
	$$
		\P (L_j > k) \le C \rho^k.
	$$
	Take $k = D\log n$ for some large $D$ which we will specify later. Note that
	$$
		\P\big( X_n - nv < -x \big)\le \P\big( T_{nv-x+k} > n\big) + \P\big( T_{nv-x+k} \le  n \mbox{ and } L_{nv - x +k} >  k \big).
	$$
	The second term is smaller than $n^{-\eps D}$, which with a proper choice of $D$ is negligible. To estimate the first term we write
	\begin{align*}
		\P\big( T_{nv-x+k}> n \big)
  			& = \P\big( 2 W_{n v-x +k} +( n v-x +k) > n \big)\\
  			& = \P\bigg(  W_{nv-x +k}  -\E W_{n v-x+k} >\frac 12 \bigg( n- nv+x  +k  -\frac{2\rho (nv-x+k)}{1-\rho}  \bigg) \bigg)\\
  			& = \P\bigg(  W_{nv-x+k }  -\E W_{n v-x+k} >  \frac{x-k}{2v} \bigg).
	\end{align*}
	Hence
	\begin{align*}
		  \frac {x^\a}{nv-x} \cdot &   \P\big( X_n - nv < -x \big) \\ &\le
		 \frac {(x-k)^\a}{nv-x+k} \cdot  \P\bigg(  W_{nv-x+k}  -\E W_{nv-x+k} >   \frac{x-k}{2v} \bigg) +o(1) = \mathcal{C}(\a).
	\end{align*}
\end{proof}

\subsection{Quantification of the environment}
	We will start with a few useful formulas for the process with immigration $\{ Z_{n}\}_{n \geq 0}$ and the process initiated by the $i$th immigrant  $\{Z_{i,n} \}_{n \geq 0}$. Firstly
	$\E_\o[Z_{i,i}]= A_{i-1}$ and by~\eqref{eq:hombr} and an appeal to independence of $\xi_k^i$'s and $Z_{i,n}$ with respect to $\P_\o$, we get
	$$
		\E_\o Z_{i,n+1} = A_n \E_\o  Z_{i,n}  \quad n \geq i.
	$$
	Whence, we infer that
	$$
		\Pi_{i-1,n-1} =\E_\o Z_{i,n} =  \prod_{j=i-1}^{n-1}A_j, \quad n \geq i.
	$$
	For the recursive formula for the quenched moments of $Z_{n}$, we go back to \eqref{eq:inhombr} and deduce that $\E_\o[Z_0]=0$ and for $n \geq 0$,
	$$
		\E_\o Z_{n+1} = A_n \E_\o Z_n + A_n.
	$$
	So that after a simple inductive argument
	$$
		Y_{n-1}= \E_\o Z_n  = \sum_{j=0}^{n-1} \Pi_{j,n-1}, \quad n \geq 1
	$$
	and $Y_0=0$. Let $\wZ^i_{k,n}$  denote the number of progeny of the $i$th immigrant, between time $n\ge k >i$, i. e.
	$$
		\wZ^i_{k,n} = \sum_{j=k}^n Z_{i,j}
	$$
	and the corresponding quenched mean, for $n \ge k>i$
	$$
		\E_\o\wZ^i_{k,n} = \wY^{i-1}_{k-1,n-1} = \sum_{j=k-1}^{n-1} \Pi_{i-1,j}.
	$$
	Finally, denote for simplicity
	$$
		\wY_{n} = \wY^0_{0,n} \quad \mbox{and} \quad \wZ_k=\wZ^1_{1,k}.
	$$
	Notice that $\wY^1_{k,n}$ and $\wY_{n-k}$ have the same distribution.

\medskip

	We defined two processes $\{Y_n\}_{n \geq 0}$ and $\{\wY_n\}_{n \geq 0}$. The first one admits the recursive formula
	$$
		Y_{n} = A_nY_{n-1}+A_n
	$$
	which is one of the most recognized Markov chains and is a particular example of the stochastic affine recursion, called also in the literature the random difference equation, or just
	the '$ax+b$' recursion. The last name reflects the fact that if we consider the pair $(A_n,A_n)$ as an element of the affine '$ax+b$' group then $Y_n$ is just the result of the action
	of this element on $Y_{n-1}$
	\begin{equation*}
  		Y_n = (A_n,A_n)\circ Y_{n-1}. 
	\end{equation*}
	In general $Y_n$ is the second coordinate of left random walk on the '$ax+b$' group, more precisely
	\begin{equation}\label{eq:100}
		Y_n = (A_n,A_n)\circ Y_{n-1} = (A_n,A_n)\circ\ldots\circ (A_0,A_0) \circ 0.
	\end{equation}
	The study of the process $\{Y_n\}_{n \geq 0}$ (usually in a more general settings with random $(A,B)$ instead of vector $(A,A)$) has a long history going back to
	Kesten~\cite{kesten1973random}, Grincevicius~\cite{grincevicius1975one}, Vervaat~\cite{vervaat1979stochastic} and others. We refer the reader to the recent
	monographs~\cite{buraczewski2016stochastic, iksanovrenewal} containing a comprehensive bibliography. 

\medskip

	The process $\{\wY_n\}_{n \geq 0}$ also can be represented in terms of the affine group. A simple calculation leads us to the following formula
	\begin{equation}\label{eq:101}
  		\wY_n =  (A_0,A_0)\circ\ldots\circ (A_n,A_n) \circ 0.
	\end{equation}
	Thus $\{\wY_n\}_{n \geq 0}$ is given as the action of the random elements $(A_j,A_j)$ but in reversed order. This explain that $\{ \wY_n\}_{n \geq 0}$ is called the backward
	process (in contrast to $\{Y_n\}_{n \geq 0}$, which is sometimes referred to as the forward process). Apart from the affine group, $\{\wY_n\}_{n \geq 0}$ has an interpretation
	in terms of Financial Mathematics, and for that reason it is very often called the perpetuity sequence.

\medskip

	Formulas \eqref{eq:100} and \eqref{eq:101} justify that for fixed $n$ random variables $Y_n$ and $\wY_n$ have the same distribution. If follows from the Cauchy ratio test that if
	$\E\log A<0$, then $\wY_n$ converges a.s. to
	$$
		\wY_\8 = \sum_{j=0}^\8 \Pi_{0,j}.
	$$
	Moreover, $\E \wY_n^{\beta} \to \E \wY_\8^{\beta}$ for any $\beta < \alpha$, for details see Section 2.3 of~\cite{buraczewski2016stochastic}. Of course this entails
	convergence in distribution of $Y_n$ to $Y_\8$.

\medskip

	The celebrated result by Kesten \cite{kesten1973random} (see also Goldie \cite{goldie1991implicit}) constitutes that $Y_\8$ has a~heavy tail.

\begin{lem}\label{lem:kesten}
	If hypothesis of Theorem~\ref{thm:main1} are  satisfied then
	$$
		\lim_{x\to\8} x^\a \P\big(\wY_\8 > x\big) = \mathcal{C}_2(\alpha),
	$$
	where
	\begin{equation}\label{eq:c1}
		\mathcal{C}_2(\alpha) = \frac{\EE \Big[ \big(\wY_\8+1\big)^{\alpha}- \wY_\8^{\alpha} \Big]}{\alpha \EE [ A^{\alpha} \log A]}.
	\end{equation}
\end{lem}

	This result was the main ingredient in \cite{kesten1975limit}. For our purposes, we need to enter deeper into the structure of both processes.
	Namely we need to understand not only the probability of exceedence of large values  by the perpetuity, but also to understand when is it most likely to happen. This problem
	was studied in \cite{buraczewski2016large, buraczewski2015pointwise}

\section{The approach}\label{sec:approach}
	Before proceed to the proof, we would like to give a reader-friendly discussion on our approach. We will state some Lemmas below and if we do not use them in the sequel, we
	restrain ourself from presenting the proof in order to keep this section as brief as possible. Define the stopping time via
	$$
		\nu = \inf_{k>0} \{ Z_k=0\}.
	$$
	After time $\nu$ the process regenerates, that is $\{ Z_{\nu +n}\}_{n \geq 0} \stackrel{d}{=}\{ Z_{n}\}_{n \geq 0}$. Due to Kesten et al.~\cite{kesten1975limit}, it is known,
	that the process regenerates exponentially fast.

\begin{lem}\label{lem:rtime}
	For some $c>0$ and $\delta >0$ one has
	$$
		\PP( \nu > k) \leq c e^{-\delta k}.
	$$
\end{lem}	

	Define the first passage time of $Z$ viz.
	$$
		\tau_t = \inf \{ n \geq 0 \: | \: Z_n >t \}.
	$$
	The  tail asymptotic of total population size, given in the next Lemma, was proved by Kesten et al.~\cite{kesten1975limit} in the case $\alpha <2$. The result can be easily
	extended to cover $\alpha \geq 2$. We provide a sketch of the argument in the next Section.

\begin{lem}\label{lem:kks}
	Under the standing assumptions
	$$
		\P\left( \sum_{k=0}^{\nu-1} Z_k > x \right)\sim \mathcal{C}_3(\alpha) x^{-\a}, \qquad x\to\8,
	$$
	where $\mathcal{C}_3(\alpha)$ is given as the finite limit of the conditional expectation
	$$
		\mathcal{C}_3(\alpha) = \mathcal{C}_2(\alpha)\lim_{t \to \infty} \EE \left[\left.Z_{\tau_t}^{\alpha} \right| \tau_t < \nu\right].
	$$
\end{lem}

	One way to approach with $\{Z_n\}_{n \geq 0}$ is via the renewal times, $\nu_0=0$, $\nu_1 = \nu$, and
	$$
		\nu_{i+1} = \inf\{ k > \nu_{i} \: | \:  Z_k = 0 \}.
	$$
	Let $N(n) = \# \{ k \: | \: \nu_k < n \}$. One has a natural way to decompose $W_n$,
	$$
		W_n = \sum_{k= 1}^{N(n)} \sum_{j=\nu_{k-1}}^{\nu_k-1} Z_k + \sum_{j= N(n)+1}^n \wZ^j_{j,\infty}.
	$$
	By an appeal to Lemma~\ref{lem:kks} we see that the first term on the right-hand side is a sum of iid terms with $\alpha$-regularly varying tails. Whence, one can expect that
	$$
		\P\left(W_n>x\right) \sim \P\left(\sum_{k=1}^{N(n)} \sum_{j=\nu_{k-1}}^{\nu_k-1} Z_k > x \right) \sim \frac{n}{\E \nu} \P\left(\sum_{k=0}^{\nu-1} Z_k > x\right)
			\sim \frac{ \mathcal{C}_3(\a) n x^{-\a}}{\E \nu}.
	$$
	This heuristic argument gives the correct order, as verified by Theorem~\ref{thm:wn}. However, due to the fluctuations of $\nu_i$'s, a rigorous argument is more complicated
	than expected. For this reason, we will proceed in a slightly different fashion.\medskip

\medskip

	Large deviations of $W_n$ are caused by deviations of the environment. Whence we need to have a good understanding of the latter. We will start with deviations of the multiplicative
	random walk $\{ \Pi_{n}\}_{n \geq 0}$. Here the answer is given by the Bahadur, Rao theorem \cite{dembo2010large}. To state it, denote
	$$
		\lambda(s) = \E[A^s] \quad \mbox{and} \quad \Lambda(s) = \log \E A^s
	$$
	with the domain $[0, \alpha_{\infty})$, where $\alpha_{\infty} = \sup\{ s: \;  \EE A^s <\infty \}$. Recall the Legendre-Fenchel transform of $\Lambda$ defined via the formula
	$$
		\Lambda^*(\rho) = \sup_{s \in \R} \{  s\rho - \Lambda(s) \}.
	$$

\begin{lem}\label{lem: petrov1}
	If the assumptions of Theorem~\ref{thm:main1} are satisfied then
	\begin{equation*}
		\PP\left(\Pi_n > e^{n\rho} \right) \sim
			\frac{c_\rho}{ \sqrt{ n}}e^{-n\Lambda^*(\rho)}
	\end{equation*}
	for some constant $c_\rho$, where $\EE \log(A)<\rho< \rho_{\infty} = \sup_{0<s<\alpha_{\infty}} \Lambda'(s)$. Moreover the convergence is almost uniform in $\rho$.
\end{lem}

	If we note that $\min \Lambda^*(\rho) = \Lambda^*(\rho_0) = \rho_0\alpha$,  where $\rho_0 = \Lambda'(\alpha)$, the result above suggests that for given $x$, the
	probability of the event $\{ \Pi_n >  x  \}$  is the largest for
	$$
		n_0 =\left\lfloor \frac{\log x}{\rho_0}\right\rfloor.
	$$
	Then we have
	$$
		\PP\left(\Pi_{n_0} > x \right) \sim \frac{C}{\sqrt{\log(x)}}x^{-\alpha}.
	$$
	Moreover, the probability that a large deviation happens outside some neighbourhood of $n_0$ in negligible. To be precise let
	$$
		m = \big\lfloor (\log x)^{1/2+\delta} \big\rfloor\quad \mbox{for small $\delta>0$}
	$$
	and consider the following Lemma.

\begin{lem}
	Let $n_1 = n_0 - m$ and $n_2 = n_0 + m$ for $m = \big\lfloor (\log x)^{1/2+\delta} \big\rfloor$ and any small $\delta>0$.
	$$
		\P\left( \sup_{k \in [n_1, n_2]}\Pi_k > x \right) \sim C x^{-\a}
	$$
	and
	$$
		\P\left( \sup_{k\notin  [n_1, n_2]}\Pi_k > x \right) =  o(x^{-\a}).
	$$
\end{lem}
	
	The first statement can be deduced from the arguments leading up to Lemma 3.9 in~\cite{buraczewski2017precise}. The second statement follows directly from Lemma~\ref{lem:ruina} stated below.

\medskip

	Deviations of $\{\Pi_n\}_{n \geq 0}$ and $\{Y_n\}_{n \geq 0}$ are closely related. The former is most likely to deviate at $n\approx n_0$ and so is the latter. More precisely, as
	proven in Section $4$ of~\cite{buraczewski2016large} a large deviation on $Y_n$ is most likely to happen for $n$ in some neighbourhood of $n_0$.

\begin{lem}\label{lem:ruina}
	Let $n_1 = n_0 - m$ and $n_2 = n_0 + m$ for $m = \big\lfloor (\log x)^{1/2+\delta} \big\rfloor$ and any small $\delta>0$. Then
	$$
		\P\left( \wY_{n_1} > x \right) = o(x^{-\a})
	$$
	and
	$$
		\P\left( \wY_\8 - \wY_{n_2} > x \right) = o(x^{-\a}).
	$$
\end{lem}

	Since the deviations of $\{\wZ_k\}_{k \geq 0}$ are mostly caused by the environment, one expects an analogue of Lemma~\ref{lem:ruina} for the total population size of a branching
	process in random environment. This is in fact the case as we have proven in Lemma 5.3 in~\cite{buraczewski2017precise}. The following Lemma is a direct consequence of Lemmas \ref{lem:6} and \ref{lem:7} given in the next section.

\begin{lem} \label{lem:negwZ}
	Let $n_1 = n_0 - m$ and $n_2 = n_0 + m$ for $m = \big\lfloor (\log x)^{1/2+\delta} \big\rfloor$ and any small $\delta>0$. Then
	$$
		\P\left( \wZ_{n_1} > x \right) = o(x^{-\a})
	$$
	and
	$$
		\P\left( \wZ^{1}_{n_2, \infty} > x \right) = o(x^{-\a}).
	$$
\end{lem}

	As a consequence, the significant part of $\wZ^k_{k,\infty}$, the total progeny of the population initiated by the $k$th immigrant conditioned on $\{ \wZ^k_{k,\infty} > x\}$, is $\wZ^{k}_{n_1+k,n_2+k}$. Whence, the
	dominant part of $W_n$ is expected to be
	$$
		\sum_{k=1}^n \wZ^k_{n_1+k,n_2+k}.
	$$
	The key feature that we will exploit is that for $n_2<|i-j|$,  $\wZ^i_{n_1+i,n_2+i}$ and $\wZ^j_{n_1+j,n_2+j}$ are independent with respect to the annealed probability $\P$.
	The strategy is to group $\wZ^i_{n_1+i,n_2+i}$'s into blocks of length $n_1$,
	\begin{equation*}
 		\W_k  = \sum_{j=(k-1)n_1}^{kn_1-1} \wZ^j_{j+n_1,j+n_2}
	\end{equation*}
	for $k=1,\ldots, p$, with $p=\lfloor n/n_1\rfloor $ and
	\begin{equation*}
  		\W_{p+1} =  \sum_{j=pn_1}^{n} \wZ^j_{j+n_1,j+n_2}
	\end{equation*}
	so that
	\begin{equation*}
		\sum_{k=1}^{p+1}\W_k = \sum_{k=1}^n \wZ^k_{n_1+k,n_2+k}.
	\end{equation*}
	We will benefit from the fact  that $\{\W_k\}_{1\leq k \leq p+1}$ forms a two-dependent sequence, i.e. for any $1\leq i \leq p-1 $, $\{\W_k\}_{1\leq k \leq i}$ and
	$\{\W_k\}_{i+3\leq k \leq p+1}$ are independent. Furthermore, $\{\W_k\}_{1\leq k \leq p}$ have the same distribution.  With this set-up, after the investigation of the asymptotic
	behaviours of $\W_1$ and the random vector $(\W_i, \W_{i \pm 1})$ we will be able to prove Theorem~\ref{thm:wn}.

\section{Preliminaries}

	One of the reasons $\{\wZ_k\}_{k\geq 0}$ has the same asymptotic behaviour of $\{\wY_k\}_{k \geq 0}$ is that in some regimes, one can successfully approximate one by the other. 	
	Throughout the article we will benefit from this phenomenon via next two Lemmas, first of which was proved in~\cite{buraczewski2017precise} as Proposition 3.1 and Corollary 3.2.

\begin{lem}\label{lem:1}
	Assume $\Lambda(\alpha) =0$ for some $\alpha>0$. Then one can find $\alpha_1$, $\alpha_2$ and $c$ such that $ 0<\alpha_1< \alpha<\alpha_2$ and for any
	$s \in [\alpha_1,\alpha_2]$ and any $n\geq 0$,
	$$
		\E Z_{1,n}^s \leq c (\lambda(s))^n.
	$$
	Moreover, if $\alpha > 1$, then 
	$$
		\E  \big|  Z_{1,n} - A_{n-1} Z_{1,n-1}\big|^\a \le C  \gamma^n,
	$$
	for some $\gamma < 1$ and a positive, finite constant $C$.
\end{lem}

	Using this Lemma, we can provide sketch of the proof for Lemma~\ref{lem:kks}.
	
\begin{proof}[Sketch of the proof of Lemma~\ref{lem:kks} for $\alpha>2$]
	The argument goes along the exact same lines as the one presented in~\cite{kesten1975limit} with the only difference that for $\alpha\geq 2$ one needs to refer to Lemma~\ref{lem:1}
	whenever a bound for $\E  \big|  Z_{1,n} - A_{n-1} Z_{1,n-1}\big|^\a$ is needed.
\end{proof}

\begin{lem}\label{lem:2}
	For any $k<n$ we have
	$$		
		\wZ_{k,n} - Z_{1,k}( \wY^k_{k,n}+1) =  \sum_{i=k+1}^n (Z_{1,i}-A_{i-1}Z_{1,i-1})(\wY^i_{i,n}+1).
	$$
\end{lem}

\begin{proof}
	We have
	\begin{align*}
	\sum_{i=k+1}^n (Z_{1,i}-A_{i-1}Z_{1,i-1})(\wY^i_{i,n}+1) & = \sum_{i=k+1}^n (Z_{1,i}-A_{i-1}Z_{1, i-1})\cdot \sum_{j=i-1}^n \Pi_{i,j} \\
			& =\sum_{j=k}^n \sum_{i=k+1}^{j+1} (Z_{1,i}-A_{i-1}Z_{1, i-1})\Pi_{i,j} \\
			& =\sum_{j=k}^n \sum_{i=k+1}^{j+1} \big(Z_{1,i}\Pi_{i,j}-Z_{1, i-1}\Pi_{i-1,j}\big) \\
			& =\sum_{j=k}^n \big( Z_{1,j+1}-Z_{1, k}\Pi_{k,j}\big)= \wZ_{k+1,n} - Z_{1,k}\wY^k_{k,n}\\
			&= \wZ_{k,n} - Z_{1,k}(\wY^k_{k,n}+1).
	\end{align*}
	This constitutes the desired formula.
\end{proof}

Next two Lemmas improve on the statement of Lemma~\ref{lem:negwZ}.

\begin{lem}\label{lem:6}
	There are constants $C,\delta>0$ such that for sufficiently large $x$
	$$
		\P\bigg( \sum_{j=1}^{n_1} \wZ^j_{j,j+n_1} > x  \bigg) \le e^{-C(\log x)^{\delta}}x^{-\a}.
	$$
\end{lem}
\begin{proof}
	Applying Lemma~\ref{lem:1} we have
	\begin{align*}
		\P\bigg( \sum_{j=1}^{n_1} \wZ^j_{j,j+n_1} > x  \bigg)
			&\le \sum_{j=1}^{n_1} \P\bigg(  \wZ^j_{j+n_1} > \frac x{2j^2}\bigg) \le \sum_{j=1}^{n_1} \P\bigg(  \wZ_{n_1} > \frac x{2j^2}\bigg)\\
			&\le \sum_{j=1}^{n_1} \P\bigg(  \sum_{k=1}^{n_1} Z_{1,k}  > \frac x{2j^2} \bigg)\le \sum_{j=1}^{n_1}
				\sum_{k=1}^{n_1} \P\bigg(   Z_{1,k}  > \frac x{4j^2k^2}\bigg)\\
			&\le Cx^{-\a -\eps} \sum_{j=1}^{n_1} \sum_{k=1}^{n_1 }  j^{2\a} k^{2\a} \E \big[ Z_{1,k}^{\a+\eps}\big] \\
			&\le C n_1^{2\a+1} x^{-\a -\eps}  \sum_{k=1}^{n_1} k^{2\a} \lambda({\a+\eps})^k \\
			&\le C n_1^{4\a+2} x^{-\a -\eps}  \lambda(\a+\eps)^{n_1}. \\
	\end{align*}
	We expand the function $\Lambda(s) = \log \lambda(s)$ into a Taylor series at point $\a$ to get
	$$
		\Lambda(\a +\eps) = \Lambda(\a) +\rho\eps + O(\eps^2).
	$$
	Take $\eps = \frac 1{\sqrt{\log x}}$ and having in mind $n_1 = n_0 -\lfloor  (\log x)^{1/2+\sigma}\rfloor$ and $n_0 = \lfloor\log x/\rho\rfloor$ write
	\begin{align*}
		\P\bigg( \sum_{j=1}^{n_1} \wZ^j_{j,j+n_1} > x  \bigg)
 			&\le C n_1^{4\a+2} x^{-\a -\eps}  e^{n_1(\rho \eps + O(\eps^2))} \\
 			&\le C n_1^{4\a+2} x^{-\a -\eps}  e^{(n_0 - (\log x)^{1/2+\sigma})(\rho \eps + O(\eps^2))} \\
 			&\le C  x^{-\a}  \cdot (\log x)^{4\a+2}   e^{-\rho (\log x)^{\sigma}} = e^{-C(\log x)^{\delta}}x^{-\a}.
	\end{align*}
\end{proof}


\begin{lem}\label{lem:7}
	There are constants $C,\delta>0$ such that
	$$
		\P\bigg( \sum_{j=1}^{n_1} \wZ^j_{j+n_2,\8} > x  \bigg) \le e^{-C(\log x)^{\delta}}x^{-\a}
	$$
\end{lem}

\begin{proof}
	We proceed in the same fashion as in the proof of  Lemma \ref{lem:6}. Applying Lemma~\ref{lem:1} we have
	\begin{align*}
		\P\bigg( \sum_{j=1}^{n_1} \wZ^j_{j+n_2,\8} > x  \bigg)
			&\le \sum_{j=1}^{n_1} \P\bigg(  \wZ^j_{j+n_2,\8} > \frac x{2j^2} \bigg)\le \sum_{j=1}^{n_1} \P\bigg(  \wZ_{n_2,\8} > \frac x{2j^2} \bigg)\\
			&\le \sum_{j=1}^{n_1} \P\bigg(  \sum_{k=n_2}^{\8} Z_{1,k}  > \frac x{2j^2}\bigg)\\
			&\le \sum_{j=1}^{n_1} \sum_{k=n_2}^{\8} \P\bigg(   Z_{1,k}  > \frac x{4j^2(k-n_2+1)^2} \bigg)\\
			&\le C x^{-\a +\eps} \sum_{j=1}^{n_1} \sum_{k=n_2}^{\8 }  j^{2(\a-\eps)} (k-n_2+1)^{2(\a-\eps)} \E \big[ Z_{1,k}^{\a-\eps}\big] \\
			&\le C n_1^{2(\a-\eps)+1} x^{-\a -\eps}  \lambda^{n_2}(\a-\eps). \\
	\end{align*}
	Recall the Taylor expansion of $\Lambda(s) = \log \lambda(s)$ at point $\a$
	$$
		\Lambda(\a -\eps) = \Lambda(\a) -\rho\eps + O(\eps^2).
	$$
	Take $\eps = \frac 1{\sqrt{\log x}}$. Since $n_1 = n_0 - \lfloor(\log x)^{1/2+\sigma}\rfloor$ and $n_0 =\lfloor \log x/\rho\rfloor$ we are allowed to write
	\begin{align*}
		\P\bigg( \sum_{j=1}^{n_1} \wZ^j_{j+n_2,\8} > x  \bigg)
 			&\le C n_1^{2(\a-\eps)+1} x^{-\a +\eps}  e^{n_2(-\rho \eps + O(\eps^2))} \\
 			&\le C  x^{-\a} (\log)^{2(\a-\eps)+1}   x^{\eps} e^{-n_0\rho \eps}  e^{- (\log x)^{1/2   +\sigma}\rho\eps} \\
			&= e^{-C(\log x)^{\delta}}x^{-\a}.
	\end{align*}
\end{proof}

From last two Lemmas, we can easily infer  Lemma~\ref{lem:negwZ}


\section{Proof of Theorem \ref{thm:wn}}
The main idea is to decompose $W_n$ into three terms
$$
	W_n = W_n^0 + \Wd_n + \Wu_n,
$$
when it is most likely, too early and too late to deviate respectively. More precisely
\begin{equation*}
  	W^0_n  = \sum_{j=1}^{n} \wZ^j_{j+n_1,j+n_2}, \qquad
	\Wd_n = \sum_{j=1}^n \wZ^j_{j+n_1-1}, \qquad
	\Wu_n = \sum_{j=1}^{n} \wZ^j_{j+n_2+1,\infty}.
\end{equation*}
As we will see below, $W_n^0$ decides about asymptotic while the other  sums are negligible and do not contribute to our final result. Denote $d_n^0 = \EE W_n^0$ if $\alpha>1$ and $d_n^0 =0$ otherwise. Define $d_n^\uparrow$ and $d_n^\downarrow$ in the same fashion.

\begin{prop}\label{prop:1}
Under the assumptions and notation  of Theorem~\ref{thm:wn}, for $\mathcal{C}_1(\alpha)= \mathcal{C}_3(\alpha)/\EE \nu$ one has
  	\begin{equation}\label{eq:23}
    		\lim_{n\to\8} \sup_{x\in\Lambda_n}\left| \frac{\P\left( W^0_n - d^0_n >x  \right)}{nx^{-\a}} -  \mathcal{C}_1(\a) \right| =0,
  	\end{equation}
  	\begin{equation}\label{eq:40}
    		\lim_{n\to\8} \sup_{x\in\Lambda_n} \frac{\P\big( \big|\Wd_n - d^{\downarrow}_n \big|>x  \big)}{nx^{-\a}}  =0,
  	\end{equation}
  	\begin{equation}\label{eq:41}
    		\lim_{n\to\8} \sup_{x\in\Lambda_n} \frac{\P\big( \big| \Wu_n - d^{\uparrow}_n \big|>x  \big)}{nx^{-\a}}  =0.
  	\end{equation}
\end{prop}

The above Proposition provides crucial estimates of large deviations of $W_n$. Its statement is an analogue of Proposition 3.9  in~\cite{Buraczewski:Damek:Mikosch:Zienkiewicz:2013}.
We will prove it  in Section~\ref{sec:proofprop}. Below we clarify how the above statement implies the main result.

\begin{proof}[Proof of Theorem \ref{thm:wn}]
	We have for fixed $\varepsilon\in (0,1)$ and any $x \in \Lambda_n$,
	\begin{align*}
		& \PP\left(W^0_n - d^{0}_n >(1+ 2\varepsilon) x\right)- \PP\left(\Wd_n - d_n^\downarrow < -\varepsilon x\right) - \PP\left(\Wu_n - d_n^\uparrow < -\varepsilon x\right) \\
		&\leq  \PP \left( W_n-d_n > x \right) \\
		& \leq   \PP\left(W^0_n - d^{0}_n >(1- 2\varepsilon) x\right)+ \PP\left(\Wd_n - d_n^\downarrow >\varepsilon x\right) + \PP\left(\Wu_n - d_n^\uparrow >\varepsilon x\right).
	\end{align*}
	Now divide everything by $nx^{-\a}$, apply Proposition~\ref{prop:1} and finally let $\varepsilon \to 0$.
\end{proof}

\section{Some properties of $W^0_n$}
In this Section we will present two results essential in the proof of Proposition~\ref{prop:1}. Notice that
\begin{equation*}
	W_n^0 = \sum_{k=1}^{p+1} \W_k,
\end{equation*}
where
\begin{align*}
 	&\W_k  = \sum_{j=(k-1)n_1}^{kn_1-1} \wZ^j_{j+n_1, j+n_2},  & k=1,\ldots, p, \quad p=\lfloor n/n_1\rfloor, \\
  	&\W_{p+1}  = W^0_n - \sum_{k=1}^p \W_k.  &
\end{align*}
Having in mind the remark concerning the dependence structure of $\{\W_k\}_{1\leq k \leq p+1}$, we will begin with an investigation of the asymptotic behaviour of $\W_1$ followed by a discussion of the behaviour of $(\W_1, \W_2, \W_3)$.

\subsection{Behaviour of $\W_1$}
 Our aim is to establish the following statement.
\begin{prop}\label{prop:m1}
	Under the standing assumptions of Theorem~\ref{thm:wn},
	$$
		\P(\W_1 > x) \sim  \frac{\mathcal{C}_3(\alpha)}{\E \nu} \; n_1  x^{-\a}.
	$$
\end{prop}

We will achieve that using next two Lemmas. Denote
$$
	n(x) = \lfloor \log \log(x)\rfloor.
$$
\begin{lem}\label{lem:p1}
	Suppose that the assumptions of Theorem~\ref{thm:wn} are in force. We have
	$$
		\P\bigg( \sum_{j=n(x)}^{\nu} \wZ^j_{j,\8} >x \bigg) =o(x^{-\a}).
	$$
\end{lem}

\begin{proof}
	We will use a very similar argument as the one presented in the proof of Lemma 3 in~\cite{kesten1975limit}. Note that $\wZ^j_\8$ is independent (with respect to the annealed
	probability $\PP$) of the event $\{ \nu \geq j \}$ since the former depends on $\o_j, \o_{j+1}, \ldots$ while the latter depends on $\o_0, \ldots , \o_{j-1}$ and $Z_1, \ldots Z_{j-1}$.
	We can write
	\begin{align*}
		\P\bigg( \sum_{j=n(x)}^{\nu} \wZ^j_{j,\8} >x \bigg) &= \P\bigg( \sum_{j=n(x)}^{\nu} {\bf 1} _{\{\nu\ge j\}}  \wZ^j_{j,\8} >x \bigg) \le
\sum_{j\ge n(x)}
				\P\bigg(  {\bf 1} _{\{\nu\ge j\}}  \wZ^j_{j,\8}  >\frac x{ 2 j^2} \bigg) \\
				&= \sum_{j\ge n(x)} \P\big( \nu> j \big) \P\bigg(  \wZ^j_{j,\8} >\frac x{ 2  j^2} \bigg) \le C x^{-\a} \sum_{j\ge n(x)} j^{2\a} \P\big( \nu> j \big)\\
				&= C x^{-\a} \E\Big[ \nu^{2\a+1} {\bf 1}_{\{\nu > n(x)\}} \Big]= o(x^{-\a}).
	\end{align*}
	The second inequality is a consequence of Lemma~\ref{lem:kks} and the fact that $\wZ_{1,\infty}^1 \leq \sum_{k=0}^{\nu -1}Z_k$.
\end{proof}

\begin{lem}\label{lem:p2}
	$$
		\P\bigg( \sum_{j=0}^{n(x)}\wZ^j_\8,\ \nu \ge n_1   \bigg) \sim\mathcal{C}_3(\alpha) x^{-\a}
	$$
\end{lem}

\begin{proof}
	We can infer the statement of the Lemma by invoking Lemmas~\ref{lem:kks}, \ref{lem:p1} and \ref{lem:7}.
\end{proof}

\begin{proof}[Proof of Proposition \ref{prop:m1}]
 	We have, by the merit of Lemma  \ref{lem:7},
\begin{align*} 	
		\P(\W_1 > x)  & = \P\bigg(      \sum_{j=0}^{n_1-1} \wZ^j_{j+n_1,\8} > x\bigg) + o(n_1 x^{-\a})\\
      & = \P\bigg(      \sum_{j=0}^{n_1-1} \wZ^j_{j+n_1,\8} > x \mbox{ and } \exists k:\; \nu_{k-1}<n_1, \nu_k-\nu_{k-1} \ge n_1 \bigg) + o(n_1 x^{-\a}).
 \end{align*}
Observe that for $k$ chosen as in the last event
 	$$
 		\sum_{j=0}^{n_1-1} \wZ^j_{j+n_1,\8} = \sum_{j=\nu_{k-1}}^{n_1-1} \wZ^j_{j+n_1,\8},
 	$$
 	since $\nu_{k-1}$  is an extinction time smaller than $n_1$, and whence $\wZ^j_{j, \nu_{k-1}} = \wZ^j_{j+n_1,\8}=0$ for $j <\nu_{k-1}$.
 Moreover such a $k$ must be unique. Denote by $\mathcal{V}$ the random set of extinction times, i.e. ${\mathcal V} = \{\nu_k\}_{k\ge 0}$. As a consequence
 of these remarks, we get
\begin{align*} 	
		\P(\W_1 > x)
& = \P\bigg(      \sum_{j=\nu_{k-1}}^{n_1-1} \wZ^j_{j+n_1,\8} > x \mbox{ and } \exists k:\; \nu_{k-1}<n_1, \nu_k-\nu_{k-1} \ge n_1 \bigg) + o(n_1 x^{-\a})\\
& = \sum_{i=0}^{n_1-1} \P\bigg( i=\nu_{k-1}\in {\mathcal V}, \; \nu_k-\nu_{k-1} \ge n_1  \mbox{ and }     \sum_{j=i}^{n_1-1} \wZ^j_{j+n_1,\8} > x \bigg) + o(n_1 x^{-\a})\\
& = \sum_{i=0}^{n_1- n(x)} \P\bigg( i=\nu_{k-1}\in {\mathcal V}, \; \nu_k-\nu_{k-1} \ge n_1  \mbox{ and }     \sum_{j=i}^{n_1-1} \wZ^j_{j+n_1,\8} > x \bigg) + o(n_1 x^{-\a}).
 \end{align*}
 Given $i$, the events $\{ i=\nu_{k-1}\in {\mathcal V} \}$ and $\{  \sum_{j=i}^{n_1-1} \wZ^j_{j+n_1,\8} > x, \nu_k- i \ge n_1  \}$ are independent. Therefore, applying consecutively Lemmas \ref{lem:p1}, \ref{lem:6}, \ref{lem:p2} and finally the (weak) renewal theorem, we have
\begin{align*} 	
		\P(\W_1 > x) & = \sum_{i=0}^{n_1- n(x)} \P\big( i = \nu_{k-1} \in {\mathcal V}\big)\P\bigg( \sum_{j=i}^{n_1-1} \wZ^j_{j+n_1,\8} > x \mbox{ and } \nu_k-i \ge n_1  \bigg) + o(n_1 x^{-\a})\\
 & = \bigg(\sum_{i=0}^{n_1- n(x)} \P\big( i  \in {\mathcal V}\big)\bigg) \; \P\bigg( \sum_{j=0}^{n(x)} \wZ^j_{j+n_1,\8} > x \mbox{ and } \nu  \ge n_1  \bigg) + o(n_1 x^{-\a})\\
 & = \E\big[ \# \{ i\le n_1-n(x):\; i\in {\mathcal V}  \}  \big] \; \P\bigg( \sum_{j=0}^{n(x)} \wZ^j_{j+n_1,\8} > x \mbox{ and } \nu  \ge n_1  \bigg) + o(n_1 x^{-\a})\\
&\sim \frac{n_1}{\E\nu} \cdot {\mathcal C}_3(\a) x^{-\alpha}.
 \end{align*}
 This completes the proof.
\end{proof}
\subsection{Asymptotic behaviour of $(\W_i, \W_{i \pm 1 })$}

Recall that $\W_i$'s via their definition depend on $x$.

\begin{prop}\label{prop:m2}
	One can find a constant $C$, such that for any $i,j$ such that $|i-j|\le2$, any $x>0$ and any $a >0$
  	$$
  		\P(\W_i>ax, \W_j>ax) \le Cn_1^{1/2+\eps}a^{-\a} x^{-\a}
  	$$
\end{prop}

\begin{proof}
	We will present a proof for $i=1$ and $j=2$. The case $i=1$ and $j=3$ can be dealt in a similar fashion. 

	We will proceed in the following fashion. Note that
  	$$
  		\P(\W_1>ax, \W_2>ax) \le \P\bigg(\sum_{j=0}^{n_1} \wZ^j_{n_1,\8}>ax, \W_2>ax\bigg).
  	$$
	In the first step we will prove that
	\begin{equation}\label{eq:dec12}
		\P\bigg(\bigg| \sum_{j=0}^{n_1} \wZ^j_{n_1, \8} -  Z_{n_1} (\wY^{n_1}_{n_1,\8}+1) \bigg| > ax\bigg)
		\le C a^{-\alpha} x^{-\alpha}.
	\end{equation}
	After that it will become evident that for our purposes it will be sufficient to estimate (in step 2)
	\begin{equation}\label{eq:dec13}
		\P\Big( Z_{n_1} (\wY^{n_1}_{n_1,\8}+1) > ax, \W_2 > ax \Big).
	\end{equation} 

\medskip

{\sc Step 1.} 
 To prove \eqref{eq:dec12}, applying Lemma \ref{lem:2} we estimate
 	\begin{align*}
		\P\bigg(\bigg| \sum_{j=0}^{n_1} \wZ^j_{n_1, \8} &-  Z_{n_1} (\wY^{n_1}_{n_1,\8}+1) \bigg| > ax\bigg)
   		\le  \P\bigg( \sum_{j=0}^{n_1} \Big|\wZ^j_{n_1, \8} - Z_{j, n_1} (\wY^{n_1}_{n_1,\8}+1)\Big| >ax\bigg)\\
   			&\le  \sum_{j=0}^{n_1}  \P\bigg(\big| \wZ^j_{n_1, \8} - Z_{j, n_1} (\wY^{n_1}_{n_1,\8}+1)\big| >\frac {ax}{2(n_1+1-j)^2}\bigg)\\
   			&\le  \sum_{j=0}^{n_1}  \P\bigg( \sum_{i=n_1+1}^\8 \big| Z_{j,i} -A_{i-1} Z_{j,i-1}\big| ( \wY^i_{i,\8}+1) >\frac {ax}{2(n_1+1-j)^2}\bigg)\\
   			&\le  \sum_{j=0}^{n_1} \sum_{i=n_1+1}^\8   \P\bigg( \big| Z_{j,i} -A_{i-1} Z_{j,i-1}\big| ( \wY^i_{i,\8}+1) >\frac {ax}{4(n_1+1-j)^2(i-n_1)^2}\bigg).
	\end{align*}
Now, if $\alpha \ge 1$, since $| Z_{j,i} -A_{i-1} Z_{j,i-1}|$ and  $\wY^i_{i,\8}$ are independent,     applying Lemma \ref{lem:kesten} and the second part of Lemma \ref{lem:1},
  we have for some $\gamma \in (0,1)$
 	\begin{align*}
		\P\bigg(\bigg| \sum_{j=0}^{n_1} \wZ^j_{n_1, \8} &-  Z_{n_1} (\wY^{n_1}_{n_1,\8}+1) \bigg| > ax\bigg)\\
&\le C \sum_{j=0}^{n_1} \sum_{i=n_1+1}^\8  (n_1+1-j)^{2\a} (i- n_1)^{2\a}a^{-\a} x^{-\a}  \E\Big[ \big| Z_{j,i} -A_{i-1} Z_{j,i-1}\big|^\a\Big] \\
 			&\le   C a^{-\a}x^{-\a} \cdot  \sum_{j=0}^{n_1}   (n_1+1-j)^{2\a} \gamma^{n_1-j} \sum_{i=n_1+1}^\8  (i- n_1)^{2\a} \gamma^{i-n_1}\\
  			&\le C a^{-\a}x^{-\a}.
	\end{align*} 
 
	If on the other hand $\alpha<1$,  we need to proceed in a slightly different way and borrow some arguments from Kesten at al. \cite{kesten1975limit}. Namely, applying the Jensen inequality, we  estimate  
 	\begin{equation*}
  		\E_\omega\Big[ \big| Z_{j,i} -A_{i-1} Z_{j,i-1}\big|^\a\Big| Z_{j,i-1} \Big]  \leq \bigg(  \E_\omega\Big[ \big| Z_{j,i} -A_{i-1} Z_{j,i-1}\big|^2 \Big| Z_{j,i-1}\Big]\bigg)^{\alpha/2} .
	\end{equation*}
	Note that with respect to $\P_\omega$, $Z_{j,i} -A_{i-1} Z_{j,i-1}$ is a sum of $Z_{j,i-1}$ independent zero mean random variables distributed as $\xi_0^{i-1}-A_{i-1}$, where $\xi_0^{i-1}$ is geometrically distributed with mean $A_{i-1}$,
	$$
		\E_\omega\Big[ \big| Z_{j,i} -A_{i-1} Z_{j,i-1}\big|^2 \Big| Z_{j,i-1}\Big] = Z_{j,i-1} \E_\omega\Big[ \big| \xi_0^{i-1} -A_{i-1} \big|^2 \Big] = Z_{j,i-1}(A_{i-1}^2+A_{i-1}).
	$$
	 Finally, invoke Lemma~\ref{lem:1} and take $\theta\in (\alpha_1 \vee \frac {\alpha}{2}, \alpha)$,
	$$
		\E_\omega \bigg(  \E_\omega\Big[ \big| Z_{j,i} -A_i Z_{j,i-1}\big|^2 \Big| Z_{j,i-1} \Big]\bigg)^{\alpha/2} \le  C \E \big[ Z_{j,i-1}^{\alpha/2} \big] \le C \E \big[ Z_{j,i-1}^\theta \big] \leq C_1 \lambda(\theta)^{j-i}.
	$$
	From here, we can apply the same arguments with $\gamma$ replaced by $\lambda(\theta)<1$.
Applying the first part of  Lemma \ref{lem:1} we conclude, as above, inequality \eqref{eq:dec12}.

\medskip

 {\sc Step 2.}  We will start with  bound for moments of $Z_k$ of order $\b<\a$, i.e. we intend to prove that 
\begin{equation}\label{eq:dec14}
\sup_k \E\big[ Z_k^\beta \big] <\8
\end{equation}
For $\alpha \le 1$ we just apply Lemma \ref{lem:1} and write
$$
 \E \big[ Z_k^\b\big] = \E \bigg(\sum_{j=0}^{k}Z_{j,k}\bigg)^\b \le \sum_{j=0}^{k}  \E \big[ Z_{j,k}^\b\big] \le C \sum_{j=0}^{\infty} \lambda(\beta)^j <\infty.
$$

If $\alpha >1$, then  
 $1=\lambda(\alpha)>\lambda(\beta)>\lambda(1)$ uniform with respect to $k$.
	By the virtue of Minkowski inequality we have
	$$
		\Big( \E Z_k^\b \Big)^{1/\b} = \bigg( \E \bigg(\sum_{j=0}^{k}Z_{j,k}\bigg)^\b \bigg)^{1/\b} \le \sum_{j=0}^{k}  \big( \E Z_{j,k}^\b \big)^{1/\b}.
	$$
	Now, with the help of Lemma~\ref{lem:1}, we write
	$$
		\sum_{j=0}^{k}  \big( \E Z_{j,k}^\b \big)^{1/\b} \le C(\beta)  \sum_{j=0}^{k} \lambda(\beta)^{(k-j)/\b}<C(\beta).
	$$
	Finally, for any given $\varepsilon$ take $\beta = \beta(\varepsilon)< \alpha$ close enough to $\alpha$.  

 Finally, 
by the
Kesten-Goldie theorem (Lemma \ref{lem:kesten}),  we estimate \eqref{eq:dec13} 
	\begin{align*}
		\P\Big( Z_{n_1} (\wY^{n_1}_{n_1,\8}+1)> ax,& \W_2 > ax\bigg)
			\le  \P\Big( (\wY^{n_1}_{n_1,\8}+1) > axn_1^{-1/(2\a)}\Big)\\& + \P\Big(  (\wY^{n_1}_{n_1,\8}+1) \le  axn_1^{-1/(2\a)},  Z_{n_1}(\wY^{n_1}_{n_1,\8}+1)> ax, \W_2 > ax\Big)\\
			&\le  C n_1^{1/2} x^{-\a} + \P\Big( n_1^{-1/(2\a)} Z_{n_1}> 1, \W_2 > x\Big)\\
			&\le  C n_1^{1/2} x^{-\a} + \P\Big(  Z_{n_1}> n_1^{1/(2\a)}\Big) \P\big( \W_2 > x\big)\\
			&\le  C n_1^{1/2} x^{-\a} + n_1^{-\b/(2\a)} \E[Z_{n_1}^\b]\cdot n_1 x^{-\a} \\
			&\le  C n_1^{1/2+\eps} x^{-\a}.
	\end{align*}
\end{proof}

\section{ Proof of Proposition \ref{prop:1}}\label{sec:proofprop}

The arguments used in the proof are  similar the proof of
Proposition 3.9 in \cite{Buraczewski:Damek:Mikosch:Zienkiewicz:2013}. However for reader's convenience we present here main steps of the proof, focusing on the arguments leading to the precise asymptotic results. We present here the proofs for $\a\in(1,2]$. For the other values of $\a$ the same scheme works, with only slight changes
(see~\cite{Buraczewski:Damek:Mikosch:Zienkiewicz:2013} for details)

\begin{proof}[Proof of Proposition \ref{prop:1}, formula \eqref{eq:23}]
	The proof strongly relies on the observation that the sum $\sum_{j=1}^p (\W_j - \E \W_j)$ is large when exactly one of the terms reaches values close to $x$, whereas contribution
	of all other factors is negligible. Below we first describe the dominant event and then justify that its complement is of smaller order. Let
  	$$
  		U = \bigg\{  \sum_{j=1}^p (\W_j - \E \W_j) >x   \bigg\}.
  	$$
  	Define $y = \frac{x}{(\log n)^{2\xi}}$ and  $z = \frac{x}{(\log n)^{\xi}}$ for   $\xi$   such that
	$$
		\xi < \frac 1{4\a}\ \mbox{and }\  2+4\xi < M.
	$$

	{\sc Step 1.} We prove that for every $\eps>0$ there is $N$ such that uniformly for all $n>N$, $x\in \Lambda_n$, the following inequality holds
	\begin{equation}\label{eq:25a}
	\begin{split}
  		(1-\eps) \frac{\mathcal{C}_3(\a)}{\E \nu} \le \frac{x^\a}{n}\cdot \P\bigg( U \cap \bigg\{ \W_k > y & \mbox{ for some $k$}, \W_i \le y \mbox{ for } i\not=k, \: 1\leq i, k\leq p \\ &\ \mbox{ and }
  			\bigg| \sum_{j\not=k} (\W_j - \E\W_j)  \bigg| \le z  \bigg\}  \bigg) \le (1+\eps)  \frac{\mathcal{C}_3(\a)}{\E \nu}.
	\end{split}
	\end{equation}
	Obviously it is sufficient to prove that for fixed $1\leq k\leq p$
	\begin{equation}\label{eq:25}
	\begin{split}
  		(1-\eps)  \frac{\mathcal{C}_3(\a)}{\E \nu} & \le \frac{x^\a}{n_1}\cdot \P\bigg( U \cap \bigg\{ \W_k > y, \W_i \le y  \mbox{ for } i\not=k \ \mbox{ and }
  			\bigg| \sum_{j\not=k} (\W_j - \E\W_j)  \bigg| \le z  \bigg\}  \bigg) \\ & \le (1+\eps)  \frac{\mathcal{C}_3(\a)}{\E \nu}.
	\end{split}
	\end{equation}
	Denote the probability above by $V_k$. We begin with upper estimates. To begin, note that one has $\E \W_k \le \frac{n_1\lambda(1)}{1-\lambda(1)}$. Indeed, since the mean of the 	
	reproduction law is $\lambda(1)$, we have
	$$
		\EE \W_k  = n_1 \EE \wZ_{n_1,n_2} \leq n_1 \EE\bigg[ \sum_{k=1}^{\infty} Z_{0,k}\bigg] =  \frac{n_1\lambda(1)}{1-\lambda(1)}.
	$$
	Thus by Proposition \ref{prop:m1}
	\begin{equation}\label{eq:30}
		V_k \le  \P\big( \W_k - \E \W_k > x-z \big) \le  \frac{\mathcal{C}_3(\a)}{\E \nu} (1+\eps) n_1 x^{-\a}.
	\end{equation}
	Lower estimates are more tedious. Firstly define
	$$
		\wt \W_k = \sum_{\tiny \substack{1\le j \le p \\ |j-k|>2}} \W_j,
	$$
	to be the sum of all $\W_j$'s independent of $\W_k$, so it is itself independent from $\W_k$. We have
	\begin{align*}
  		\frac{x^\a}{n_1}\cdot V_k & \ge   \frac{x^\a}{n_1} \cdot  \P \Big( \W_k - \E\W_k > x+z, |\wt \W_k - \E \wt \W_k| \le z-8y, \W_i\le y, i\not=k  \Big)\\
 			& =    \frac{x^\a}{n_1} \cdot \P \big( \W_k - \E\W_k > x+z \big)\\
 			& - \frac{x^\a}{n_1} \cdot \P \Big( \big\{ \W_k - \E\W_k > x+z\big\} \cap \big\{  |\wt \W_k- \E \wt \W_k| > z-8y\ \mbox{ or } \W_i >  y\ \mbox{for some } i\not=k  \big\}\Big)
	\end{align*}
	Proposition \ref{prop:m1} provides us with the lower bound for the first term. Assuming we can justify that the second term is negligible, i.e.
	\begin{equation}\label{eq:26}
	\begin{split}
		\P \Big( \big\{ \W_k - \E\W_k &> x+z\big\}\\& \cap \big\{  |\wt \W_k-\E \wt \W_k| > z-8y\ \mbox{ or } \W_i >  y\ \mbox{for some } i\not=k  \big\}\Big) = o(n_1 x^{-\a}),
	\end{split}
	\end{equation}
	we obtain
	\begin{equation}\label{eq:31}
		V_k \ge  \frac{\mathcal{C}_3(\a)}{\E \nu} (1-\eps) n_1 x^{-\a}.
	\end{equation}
	To prove \eqref{eq:26} we need to bound separately  two factors and establish:
	\begin{align}
		\label{eq:32}  I  = \P \big(  \W_k - \E\W_k > x+z \mbox{ and }  \W_i >  y\ \mbox{for some } i\not=k \big\} \big) & = o(n_1 x^{-\a}), \\
 		\label{eq:33} II  = \P \big( \W_k - \E\W_k > x+z \mbox{ and }  |\wt \W_k - \E\wt\W_k| > z-8y \mbox{ and } \W_i \le   y,   i\not=k  \big) & = o(n_1 x^{-\a}).
	\end{align}
	To estimate $I$ we apply Propositions \ref{prop:m1},  \ref{prop:m2} with $\sigma>0$ sufficiently small, $a = (\log n)^{-2\xi}$ and use independence of $\W_i$ and $\W_k$ for $|i-k|>2$:
	\begin{align*}
  		I & \le \sum_{i\not= k }\P\big( \W_k > x \ \mbox{ and }\ \W_i>y  \big)\\
 		& \le \sum_{ 0<|i-k|\le 2 }\P\big( \W_k > y \ \mbox{ and }\ \W_i>y  \big) +\sum_{ 2 <  |i-k| }\P\big( \W_k > x \ \mbox{ and }\ \W_i>y  \big)  \\
 		&\le C n_1^{1/2+\sigma} y^{-\a} + C p\cdot n_1 x^{-\a}\cdot n_1 y^{-\a}\\
 		&\le C n_1 x^{-\a} \big( n_1^{\sigma-1/2} (\log n)^{2\xi \a} + n(\log n)^{4\xi \a} x^{-\a}  \big).
	\end{align*}
	Now it is just sufficient to justify that the expression in the brackets is tends to zero, but this follows directly from our assumptions on $\xi$ and the definition of the domain $\Lambda_n$.

	To bound $II$ we first use the independence of $\W_k$ and $\wt \W_k$ and write
	$$
		II \le  \P \big( \W_k - \E\W_k > x+z \big)\P\big(  |\wt \W_k-\E \wt \W_k| > z-8y\ \mbox{ and } \W_i \le   y\ \mbox{for all } |i-k| > 2  \big).
	$$
	In view of Proposition \ref{prop:m1} it is sufficient to prove
	\begin{equation}\label{eq:27}
 		\P\big(  |\wt \W_k-\E \wt \W_k| > z-8y\ \mbox{ and } \W_i \le   y\ \mbox{for all } |i-k| > 2  \big) = o(1),\qquad n\to\8
	\end{equation}
	For this purpose we need   the Prokhorov inequality (see Petrov \cite{petrov1995limit}, p.~77):
{\it
		Let  $(X_n)$ be a sequence of independent random variables and denote their partial sums by $R_n=X_1+\cdots +X_n$. We write $B_n={\rm var}(R_n)$.
		Assume that the $X_n$'s are centered, $|X_n|\leq y$ for all $n\ge 1$ and some $y>0$. Then}
		\begin{equation}\label{eq:prokhorov}
 			\P \{ R_n\geq x\} \leq \exp \Big\{ -\frac{x}{2\,y}{ \rm arsinh }\big(\frac{xy}{2\,B_n}\big)\Big\}\,,\quad x>0\,.
		\end{equation}

	The Prokhorov inequality requires the random variables to be bounded and independent. To reduce our problem to this setting we use $2$-dependence of the sequence
	$\{\W_i\}_{1\leq i \leq p+1}$ and we decompose the sum $\wt \W_k$ into sum of three blocks, each consisting of i.i.d. random variables
	\begin{align*}
 		\P\big(   & |\wt \W_k-\E \wt \W_k| > z-8y\ \mbox{ and } \W_i \le   y\ \mbox{for all } |i-k| > 2  \big)\\
			& \le \P\bigg( \bigg|
 				\bigg(
 					\sum_{\tiny \substack{ 1\le j \le p\\j\in \{1,4,7,\ldots\} \\ |j-k|>2 }}
 					+ \sum_{\tiny \substack{1\le j \le p\\j\in \{2,5,8,\ldots\} \\ |j-k|>2 }}
 					+ \sum_{\tiny \substack{1\le j \le p\\j\in \{3,6,9,\ldots\} \\ |j-k|>2 }}
  				\bigg) \big(\W_j - \E \W_j \big)
  				\bigg| > \frac z2 \ \mbox{ and } \W_j \le y, j\not=k \bigg)\\
			& \le 3 \P\bigg( \bigg| \sum_{\tiny \substack{1\le j \le p\\j\in \{1,4,7,\ldots\} \\ |j-k|>2 }} \big(\W_j - \E \W_j \big) \bigg| > \frac z6 \ \mbox{ and } \W_j \le y \bigg).
	\end{align*}
	Next we reduce the problem to bounded random variables by introducing the truncations
	$$
		\W^y_j = \W_j {\bf 1}_{\{\W_j \le y\}}.
	$$
	We prove that the remaining part, that is $\W_j-\W^y_j$ is negligible. Applying twice the Minkowski inequality, we estimate the $\a$ norm of $\W_j$ with the help of Lemma~\ref{lem:1}
	\begin{align*}
  		\big(\E \W_j^\a\big)^{\frac 1\a}
  			&  = \bigg(\E\bigg(\sum_{i=0}^{n_1} \wZ^i_{i+n_1,i+n_2}  \bigg)^\a \bigg)^{1/\a} \le \sum_{i=0}^{n_1} \Big( \E\big( \wZ^i_{i+n_1,i+n_2}  \big)^\a \Big)^{1/\a}\\
  			&  = n_1  \bigg( \E \bigg(   \sum_{k=n_1}^{n_2} Z_{0,k}\bigg)^\a \bigg)^{1/\a} \le n_1  \sum_{k=n_1}^{n_2} \big( \E  Z^\a_{0,k}\big)^{1/\a} \le C n_1 m.
	\end{align*}

	Therefore, by the H\"older inequality
	\begin{align*}
  		p\E\big[\W_j {\bf 1}_{\{\W_j > y\}}\big]
  			& \le p \big( \E\W_j^\a \big)^{1/\a}\P(\W_j > y)^{1-1/\a} \le Cpn_1 m\cdot n_1^{1-1/\a} y^{1-\a}\\
  			&\le C (\log x)^{\frac 32 + \sigma - \frac 1\a + (\a-1)2\xi}nx^{1-\a} = o(x),
	\end{align*}
	where  the last inequality follows for our assumptions on $\xi$ and $\Lambda_n$. To see that, consider two possibilities, first of which is $x >n$. Then, if $n$ is large enough
	$ x > \log(x)^M n^{1/\alpha}$ and as a consequence
	$$
		x^{-\alpha} \leq \log(x)^{-\alpha M} n^{-1}
	$$
	and so
	$$
		(\log x)^{\frac 32 + \sigma - \frac 1\a + (\a-1)2\xi}nx^{1-\a} \leq x (\log x)^{\frac 32 + \sigma - \frac 1\a + (\a-1)2\xi} \log(x)^{-\alpha M} = o(x)
	$$
	due to constraints imposed on $\xi$. In the second case, i.e. $x <n$ we have
	$$
		(\log x)^{\frac 32 + \sigma - \frac 1\a + (\a-1)2\xi}nx^{1-\a} \leq (\log n)^{\frac 32 + \sigma - \frac 1\a + (\a-1)2\xi} \log(n)^{-\alpha M}x= o(x),
	$$
	for the same reason as before. Hence, it is sufficient to estimate
	$$
		\P\bigg(\bigg| \sum_{\tiny \substack{1\le j \le p\\j\in \{1,4,7,\ldots\} \\ |j-k|>2 }} \big(\W^y_j - \E \W^y_j \big) \bigg| > \frac z7  \bigg).
	$$
	We use the Prokhorov inequality \eqref{eq:prokhorov} with
	\begin{align*}
  		X_i & =\W_i^y - \E\W^y_i\\
  		B_p &= p{\rm var \W_i^y} \le p y^{2-\a} \E\W_1^\a \le Cpy^{2-\a} n_1^\a m^\a
	\end{align*}
	and considering two possibilities $x<n$ and $x \geq n$ in combination with the fact that $x \in \Lambda_n$ we obtain
	\begin{align*}
		\P\bigg(\bigg| \sum_{\tiny \substack{1\le j \le p\\j\in \{1,4,7,\ldots\} \\ |j-k|>2 }} \big(\W^y_j - \E \W^y_j \big) \bigg| > \frac z7  \bigg)
  			& \le e^{-\frac{Cz}{y}  \cdot {\rm arcsinh} (\frac{zy}{2B_p}) } \le C \bigg( \frac{2B_p}{zy}  \bigg)^{C (\log n)^{\xi}}\\
  			&\le C \Big( n(\log x)^{(1/2+\sigma)\a} \log(n)^{2\a\xi-\xi} x^{-\a}\Big)^{C(\log n)^{\xi}}= o(1).
	\end{align*}
	This completes the proof of \eqref{eq:33}, which together with \eqref{eq:32} entails \eqref{eq:26}. Combining \eqref{eq:30} with \eqref{eq:31} we
	obtain \eqref{eq:25} and hence \eqref{eq:25a}.

\medskip

	{\sc Step 2.} Now we consider the remaining cases, not treated in the first step, which are of smaller order. We begin with  the event when all $\W_i$, except $\W_k$, are small,
	despite this, their sum is large. That is we intend to show
	\begin{equation}\label{eq:261}
	\begin{split}
 		\P\bigg( U \cap \bigg\{  \W_k > y \mbox{ for some }k, &\W_i \le y \mbox{ for } i\not=k \\ &\mbox{ and }
  		\bigg| \sum_{j\not=k} (\W_j - \E\W_j)  \bigg| > z  \bigg\}  \bigg) = o(n x^{-\a})
	\end{split}
	\end{equation}
	As previously it is sufficient to prove for fixed $k$
	\begin{equation}\label{eq:262}
	\begin{split}
 		\P\bigg( U \cap \bigg\{ \W_k > y, \W_i \le y \mbox{ for } i\not=k \ \mbox{ and }
  		\bigg| \sum_{j\not=k} (\W_j - \E\W_j)  \bigg| > z  \bigg\}  \bigg) = o(n_1 x^{-\a})
	\end{split}
	\end{equation}
	We estimate this probability by
	$$
		\P(\W_k > y) \cdot \P \big(  |\wt \W_k-\E \wt \W_k| > z-8y \ \mbox{ and } \W_i \le  y\ \mbox{for } i\not=k  \big)
	$$
	and then we proceed exactly as in the first step, that is we apply Proposition \ref{prop:m1} and to bound the second term the Prokhorov inequality~\eqref{eq:prokhorov}. We omit details.

\medskip

	{\sc Step 3.} Next we consider the event when all $\W_j$'s are smaller than $y$ and then again the Prokhorov inequality~\eqref{eq:prokhorov} yields
	\begin{equation}\label{eq:271}
	\begin{split}
 		\P\big( U \cap \{ \W_i \le y \mbox{ for all } i\}\big) = o(n x^{-\a})
	\end{split}
	\end{equation}

\medskip

	{\sc Step 4.} Finally when at least two $\W_j$'s are larger than $y$, the same arguments as in the proof of \eqref{eq:32} entail
	\begin{equation}\label{eq:281}
	\begin{split}
 		\P\big( U \cap \{ \W_i >y, \W_j > y \mbox{ for some  } i\not= j\} \big) = o(n x^{-\a})
	\end{split}
	\end{equation}
	We refer the reader to the proof of Proposition 3.9 in  \cite{Buraczewski:Damek:Mikosch:Zienkiewicz:2013} for more details.
\end{proof}

\begin{proof}[Proof of Proposition \ref{prop:1}, formula \eqref{eq:40}]
  	We proceed as in the proof of formula \eqref{eq:23}. Recall
  	$$
  		\Wbd_k = \sum_{j=(k-1)n_1}^{kn_1-1}\wZ^j_{j+n_1-1}.
  	$$
  	Then $\Wbd_k$ are identically distributed and one dependent, i.e. if $|i-j|>1$, then $\Wbd_i$ and $\Wbd_j$ are independent. We have
  	\begin{align*}
    		\P\big( \big|  \Wd_n - \E \Wd_n    \big|> x \big) & \le \P\big( \Wbd_k > y \mbox{ for some }k\big)\\
		&+    \P\big( \big|  \Wd_n - \E \Wd_n    \big|> x \mbox{ and } \Wbd_k \le y \mbox{ for all }k\big)
  	\end{align*}
	To bound the first term we just use Lemma \ref{lem:6}
	\begin{align*}
		\P\big( \Wbd_k > y \mbox{ for some }k\big) & \le \sum_{k=1}^{p+1} \P\big( \Wbd_1 > y \big)\\
		&\le p e^{- C(\log y)^{\delta}} y^{-\a}\\
		&\le n x^{-\a} \cdot n_1^{-1} (\log x)^{2\xi} e^{-C (\log x)^{\delta}}= o(n x^{-\a}).
	\end{align*}
	And for the second term we use the Prokhorov inequality \eqref{eq:prokhorov}.
\end{proof}

\begin{proof}[Proof of Proposition \ref{prop:1}, formula \eqref{eq:41}]
	We would like to repeat the procedure from previous proofs of \eqref{eq:23} and \eqref{eq:40}.
However this time we need to proceed more carefully, because all
  	the factors in the sum defining $\Wu_n$ are dependent and we cannot use directly the block decomposition into sum of i.i.d. terms.\\

  	To overcome this difficulty we cut the factors $\wZ^j_{j+n_2,\8}$ at some place. Let $n_3 = D\log x$, where $D$ is a large constant satisfying $D > \frac {\a-1}{|\log \E A|}$.  
  We are going to prove
  	\begin{equation}\label{eq:42}
    		\P\bigg(  \bigg| \sum_{j=1}^{n-n_3} \wZ^j_{j+n_3+1, n }- z_n \bigg| >x \bigg) \le cnx^{-\a -\delta}
  	\end{equation}
  	for some $\delta>0$, where
  	$$
  		z_n = \EE \sum_{j=1}^{n-n_3} \wZ^j_{j+n_3+1, n }.
  	$$
	We have
   	\begin{align*}
     		\E  \big[ \wZ^j_{j+n_3+1,n}\big]  & \le \E \bigg[ \sum_{k=n_3}^\8 Z_{0,k} \bigg] =  \sum_{k=n_3}^\8 \E \big[  Z_{0,k} \big] \\
     		&\le C \lambda(1)^{n_3} \le C x^{D \log \lambda(1)} \le C x^{ 1-\a-\delta}
   	\end{align*}
   	and hence
	$$
		\P\bigg(  \bigg| \sum_{j=1}^{n-n_3} \wZ^j_{j+n_3+1, n} - z_n \bigg| >x \bigg) \le \frac 2x \sum_{j=1}^{n-n_3} \E\big[ \wZ^j_{j+n_3+1,n}\big] \le C n x^{-\a-\delta}
    	$$
    	Thus
   	$$
   		\lim_{n\to\8} \sup_{x\in \Lambda_n} \P\bigg(\bigg|\sum_{j=0}^{n-n_2} \wZ^j_{j+n_2+1, j+n_3} -  z_n\bigg| > x  \bigg) =0
   	$$
   	and now we can proceed as previously. Define
   	$$
   		\Wbu_k = \sum_{j=(k-1)n_1}^{kn_1-1} \wZ^j_{j+n_2, j+n_3}.
   	$$
   	Then $   \Wbu_k$ have the same distribution and $   \Wbu_i$, $   \Wbu_j$ are independent if $|i-j|> \rho D + 1$. We can repeat previous arguments.
\end{proof}

\bibliographystyle{plain}
\bibliography{dbura_pdysz_PLDforRWRE_bib}

\end{document}